\documentclass[12pt]{article}
\usepackage{e-jc}

\usepackage{amsthm,amsmath,amssymb}
\usepackage{graphicx}

\usepackage[colorlinks=true,citecolor=black,linkcolor=black,urlcolor=blue]{hyperref}

\theoremstyle{plain}
\newtheorem{theorem}{Theorem}
\newtheorem{lemma}[theorem]{Lemma}
\newtheorem{corollary}[theorem]{Corollary}
\newtheorem{proposition}[theorem]{Proposition}

\theoremstyle{definition}
\newtheorem{definition}[theorem]{Definition}
\newtheorem{example}[theorem]{Example}

\newtheorem{question}[theorem]{Question}

\theoremstyle{remark}
\newtheorem{remark}[theorem]{Remark}

\title{\bf On the rank function of a differential poset}

\author{Richard P. Stanley\\
\small Department of Mathematics\\[-0.8ex]
\small MIT\\[-0.8ex]
\small Cambridge, MA 02139-4307\\
\small\tt rstan@math.mit.edu\\
\and
Fabrizio Zanello\\
\small Department of Mathematical  Sciences\\[-0.8ex]
\small Michigan Tech\\[-0.8ex]
\small Houghton, MI  49931-1295\\
\small\tt zanello@math.mit.edu
}

\date{\dateline{November 19, 2011}{April 15, 2012}}

\begin{document}
\maketitle

\begin{abstract}
We study $r$-differential posets, a class of combinatorial objects introduced in 1988 by the first author, which gathers together a number of remarkable combinatorial and algebraic properties, and generalizes important examples of ranked posets, including the Young lattice. We first provide a simple bijection relating differential posets to a certain class of hypergraphs,  including all finite projective planes, which are shown to be naturally embedded in the initial ranks of some differential poset. As a byproduct, we prove the existence, if and only if $r\geq 6$, of $r$-differential posets nonisomorphic in any two consecutive ranks but having the same rank function. We also show that the Interval Property,  conjectured by the second author and collaborators for several sequences of interest in combinatorics and combinatorial algebra, in general fails for differential posets.  In the second part, we prove that the rank function $p_n$ of any arbitrary $r$-differential poset has nonpolynomial growth; namely, $p_n\gg n^ae^{2\sqrt{rn}},$ a bound very close to the Hardy-Ramanujan asymptotic formula that holds in the  special case of Young's lattice. We conclude by posing several open questions.

\end{abstract}

\section{Introduction}

The first author introduced $r$-differential posets  in 1988 (\cite{St}), as a class of ranked posets enjoying a number of important combinatorial and algebraic properties. The best known examples of differential posets are the \emph{Young lattice} $Y$, and the \emph{Fibonacci $r$-differential posets} $Z(r)$; these latter are a generalization of the Young-Fibonacci lattice, which corresponds to the case $r=1$. Many authors have since   provided interesting generalizations, in several directions, of the concept of a differential poset; see \cite{BR,Fo1,Fo2,Ho,Lam1,Lam2,St2}.

Despite  their relevance, a number of basic  questions concerning these posets remain open. For instance, it is unknown what the structure of a differential poset $P$ can be, even in low rank; it is still a conjecture that the rank function, $p_n$, of $P$ must be \emph{strictly} increasing (that it \emph{weakly} increases was proved in \cite{St}, by means of a linear algebra argument); and, no nontrivial lower bounds were known until now on the growth of $p_n$ (see \cite{B} for a sharp upper bound). The object of this note is to start  addressing some of the above or related structural questions, by also drawing some interesting,  natural connections to other combinatorial areas.

The first part of our paper is devoted to proving a natural bijection relating differential posets to such other combinatorial  objects as hypergraphs, Steiner systems and finite projective planes. These latter are  shown to be all naturally embedded in the first two ranks of some differential poset. It follows that a characterization of the initial portion of differential posets would already imply, for instance, a complete classification of finite projective planes, one of the major open questions in geometry. As a consequence of our correspondence, we  establish the existence, if and only if $r\geq 6$, of $r$-differential posets with the same rank function but nonisomorphic in any two consecutive ranks, and we also  show that the Interval Property ---  a (still mostly conjectural) property of interest in other areas of combinatorial algebra and combinatorics --- in general fails for differential posets. 

In the second part of the paper, we prove that, for any  $r$-differential poset, its rank function $p_n$ has nonpolynomial growth. The conjectural lower bound for $p_n$ is given by the rank function $p_r(n)$ of $Y^r$, the Cartesian product of $r$ copies of the Young lattice. This latter function, using the Hardy-Ramanujan asymptotic formula for the partition function $p(n)=p_1(n)$, and some standard properties of the Riemann zeta function, can be proved to be asymptotically equal to $Cn^{\alpha }e^{\pi\sqrt{2rn/3}},$ for suitable constants $C$ and $\alpha$. Our main result is that, for any arbitrary $r$-differential poset, $p_n\gg n^ae^{2\sqrt{rn}},$ for some constant $a$. Since $\pi\sqrt{2/3}=2.56...$, this bound is  close to being optimal, and in the much greater generality of arbitrary differential posets.

We conclude our note by listing some of the main open questions on the rank function of a differential poset.

We refer to \cite{St,St1} for any unexplained terminology, and for the basic facts of the theories of posets and differential posets. 

\begin{definition}\label{diff}
Let $r$ be a positive integer, and  $P=\bigcup_{n\geq 0}P_n$  a locally finite, graded poset with a unique element of rank zero. Then $P$ is  \emph{$r$-differential} if:
\begin{enumerate}
\item[(i)]  two distinct elements $x,y\in P$ cover  $k$ common elements if and only if they are covered by  $k$ common elements; and
\item[(ii)] an element $x\in P$ covers $m$ elements if and only if it is covered by $m+r$ elements.
\end{enumerate}
\end{definition}

The next lemma will be of great (and often implicit) use. Its proof is easy to see by induction, given that a differential poset has a unique least element (see \cite{St1}, Proposition 1.2).

\begin{lemma}\label{cover}
The integer $k$ of Definition \ref{diff}(i) equals 0 or 1. In other words, any two distinct elements of a differential poset can cover, or be covered by, at most one common element.
\end{lemma}

The following construction, which generalizes that of the Fibonacci $r$-differential poset $Z(r)$, is due to D. Wagner (see \cite{St,St1}), and is still the main general tool known today to construct new differential posets. In particular, it was used by J.B. Lewis \cite{Le} to prove the existence of uncountably many nonisomorphic differential posets.

Let $P'=\bigcup_{n=0}^jP_n$ be an $r$-differential poset \emph{up to rank $j$}; that is, $P'$ is a graded poset of rank $j$ with a unique element of rank zero, satisfying Definition \ref{diff} in all ranks up to $j-1$. Let $P_{j-1}=\{y_1,\dots ,y_t\}$. Wagner's construction adds a next rank, $P_{j+1}$, to $P'$, as follows. For each $i=1,\dots ,t$, place an element $w_i$ in $P_{j+1}$ such that it covers some element $x\in P_j$ if and only if $x$ covers $y_i$. Also, for each  $x\in P_j$, place $r$ new elements in $P_{j+1}$, each covering only $x$. It is easy to check that the new poset $P'\bigcup P_{j+1}$ thus defined is $r$-differential up to rank $j+1$. Infinitely iterating this construction yields an $r$-differential poset, $P=\bigcup_{n\geq 0}P_n$. 

It immediately follows from Wagner's construction that the rank function of $P$ is given by $p_n=rp_{n-1}+p_{n-2}$, for all indices $n\geq j+1$.

\section{(Negative) structural results}

In this section, we  present a very natural correspondence between differential posets and a suitable family of hypergraphs, which  allows us to shed some more light on the complicated structure of a differential poset. In the concluding section of \cite{St}, the first author asked for a classification of differential posets, and pointed out how difficult this problem is by  noticing that, in light of Wagner's construction, it is even ``unlikely that [the problem] has a reasonable answer''. We  see in this section  that, in fact, even a characterization of  ranks one and two of a differential poset would imply, as particular cases, results of remarkable interest in finite geometry and design theory --- including a classification of all finite projective planes. 

As a consequence of our correspondence, we  also show that there exist $r$-differential posets that are nonisomorphic in any two consecutive ranks but have the same rank function, if and only if $r\geq 6$. Further, we prove that the Interval Property fails for the set $\mathbb X$ of all rank functions of differential posets, by showing that $1,4,14,60,254,p_5,p_6,\dots $ and $1,4,17,60,254,p_5,p_6,\dots $ belong to $\mathbb X$ for a suitable choice  of the $p_i$,  while no sequence in $\mathbb X$ can begin $1,4,16,\dots $. 

Recall that a \emph{hypergraph} $H$ is a pair $H=(V,E)$, where $V=\{1,\dots, r\}$ is a vertex set, and $E$ is a  collection of nonempty subsets of $V$, called the  \emph{hyperedges} of $H$.  Borrowing a term from algebraic combinatorics, we define the \emph{dimension} of a hyperedge $F$ of $H$ as $\dim F = |F| -1$. We say that $H$ is \emph{$k$-uniform} if all of its hyperedges have dimension $k-1$. The \emph{$(r-1)$-simplex} is the hypergraph over $V$ whose only hyperedge is $V$ itself. 

The \emph{complete graph} over the vertex set $V$ is the graph, denoted by $K_r$, where all vertices of $V$ are connected by an edge. Therefore, $K_r$ can  be identified with the 2-uniform hypergraph $H$ over $V$ where every cardinality two subset of $V$ is a hyperedge of $H$. Finally, following \cite{St1}, given a graded poset $P=\bigcup_{n\geq 0}P_n$ and a set $S$ of nonnegative integers, define the \emph{rank-selected subposet} $P_S$ of $P$ as the poset of elements of $P$ whose ranks belong to $S$; that is, $P_S=\bigcup_{n\in S}P_n$. If $S$ is the interval $\{a,a+1,\dots ,b\}$, one usually writes $S=[a,b]$.

We have:

\begin{proposition}\label{simp}
\begin{enumerate}
\item[(1)] The rank-selected subposets $P_{[1,2]}$ of $r$-differential posets are in bijection with the hypergraphs $H$ over $V=\{1,\dots, r\}$, such that  all hyperedges of $H$ have positive dimension and any 2-subset of $V$ is contained in exactly one hyperedge of $H$.
\item[(2)] Fix an $r$-differential poset $P$ with rank function $\textbf{p}{\ }: {\ }p_0=1,p_1=r,p_2,\dots $, and let $T_1$ denote the sum of the dimensions of the hyperedges of the hypergraph associated with $P_{[1,2]}$, as from part (1). Then $$p_2=r(r+1)-T_1.$$
\item[(3)] The maximum value that $p_2$ may assume for an $r$-differential poset is $r^2+1$; it is achieved by  $Z(r)$, whose associated hypergraph is the $(r-1)$-simplex.
\item[(4)] If $r>1$, the next largest value  $p_2$ may assume is $r^2-r+3$, corresponding to a hypergraph with one $(r-2)$-dimensional hyperedge and $r-1$ 1-dimensional hyperedges.
\item[(5)] The smallest value  $p_2$ may assume is $\binom{r+2}{2}-1$; it is achieved by  $Y^r$ (the Cartesian product of $r$ copies of the Young lattice $Y$), whose associated  hypergraph is $K_r$.
\end{enumerate}
\end{proposition}

\begin{proof}
(1). Given an $r$-differential poset $P$, identify the elements of $P_1$ with the integers $1,\dots, r$, and an element $x\in P_2$ with the subset $\{a_1,\dots ,a_t\}$ of $V=\{1,\dots, r\}$, if and only if $t\geq 2$ and $a_1,\dots ,a_t$  are the elements of $P_1$ covered by $x$. (Thus, notice that we are not considering here elements of rank 2 covering a unique element.) That such subsets of $V$ are the hyperedges of a hypergraph with the desired properties is immediate to check, using the two axioms of Definition \ref{diff}. The fact that this is a bijection follows from axiom (ii), since every rank one element of an $r$-differential poset $P$ is covered by exactly $r+1$ elements. This easily determines uniquely the structure of $P_{[1,2]}$ corresponding to a hypergraph as in the statement.

(2). Since every element of $P_1$ is covered by exactly $r+1$ elements,  $p_2$ equals $r(r+1)$ minus a sum of contributions coming from the elements of $P_2$ covering at least two elements, i.e., by part (1), the hyperedges of the hypergraph $H$ associated with $P_{[1,2]}$. But it is immediate to see that each hyperedge of $H$ contributes to that sum by exactly  its cardinality minus 1, that is, its dimension.

(3) and (4). If $r=1$, part (3) is trivial. Hence suppose $r>1$. In order to prove both statements, it is enough to show that the value of $T_1$ for the $(r-1)$-simplex is $r(r+1)-(r^2+1)=r-1$, and  the value of $T_1$ for any other admissible hypergraph is at least $r(r+1)-(r^2-r+3)=2r-3$. 

The first fact is entirely obvious. As for the second, it is a simple exercise to check, using axiom (i)  of Definition \ref{diff}, that if an admissible hypergraph contains a hyperedge of cardinality $c$, then it must contribute to $T_1$ by at least $(c-1)+c(r-c)$. But $(c-1)+c(r-c)\geq 0$ if and only if
\begin{equation}\label{ccc}
c^2-(r+1)c+2r-2=(c-(r-1))(c-2)\leq 0,
\end{equation}
which is true for all admissible  hypergraphs except the $(r-1)$-simplex. 

Since equality (\ref{ccc}) is satisfied for $c=r-1$, it follows that $p_2=r^2-r+3$ can (only) be achieved by a hypergraph $H$ with a unique $(r-2)$-dimensional hyperedge and $r-1$ hyperedges of dimension one; that is, up to isomorphism, $H$ is the hypergraph with hyperedges: $\{1,2,\dots  ,r-1\}, \{1,r\}, \{2,r\}, \dots ,\{r-1,r\}$. 

(5). A simple computation gives that the value of $p_2$ for $Y^r$ is $p_2=\binom{r+2}{2}-1$, while the fact that the  hypergraph associated with $Y^r$ is $K_r$ follows by noticing that no rank two element in $Y^r$ can cover more than two elements. In order to show  that $Y^r$ yields the smallest possible value for $p_2$, or equivalently, the largest possible value for  $T_1=r(r+1)- p_2$, consider any admissible hypergraph having a hyperedge of cardinality $c$. It is easy to see that this hyperedge can be replaced by $\binom{c}{2}$ 1-dimensional hyperedges covering the same set of $c$ vertices, and  leaving the rest of the hypergraph unchanged. Since $\binom{c}{2}\cdot 1\geq 1\cdot (c-1)$ for any $c\geq 2$, the result  easily follows by induction.
\end{proof}

\begin{example}\label{4}
When $r=2$, it is clear that there is (up to isomorphism) a unique possible rank-selected subposet $P_{[1,2]}$ for an $r$-differential poset, where the two rank one elements are covered by three elements each, one of which is in common. The hypergraph associated with $P_{[1,2]}$ is the 1-simplex, which is isomorphic to $K_2$. Notice that, when $r=2$, $r^2+1= r^2-r+3=\binom{r+2}{2}-1=5,$ which is indeed the only possible value of $p_2$.

As for $r=3$, there are, according to Proposition \ref{simp}, part (1), only two nonisomorphic admissible  hypergraphs: the 2-simplex, giving $p_2=10$; and $K_3$, giving $p_2=9$.

When $r=4$, there are up to isomorphism three admissible hypergraphs over $V=\{1,2,3, 4\}$: the 3-simplex; $K_4$; and (see Figure 1) the hypergraph $H $ whose hyperedges are $\{1,2,3\}$, $\{1,4\}$, $\{2,4\}$, $\{3,4\}$. The corresponding values of $p_2$, which are therefore the only ones allowed for a 4-differential poset, are 17, 14, and 15, respectively.

When $r=5$, there are up to isomorphism five admissible  hypergraphs. The corresponding values of $p_2$ are: 20, 21, 22, 23, 26. 

As for $r=6$, the admissible nonisomorphic  hypergraphs are ten. The corresponding values of $p_2$ are: 27, 28, 29 (occurring twice), 30 (twice), 31 (twice), 33, 37.
\end{example}
\begin{figure}[h!t]
\includegraphics[scale=0.80]{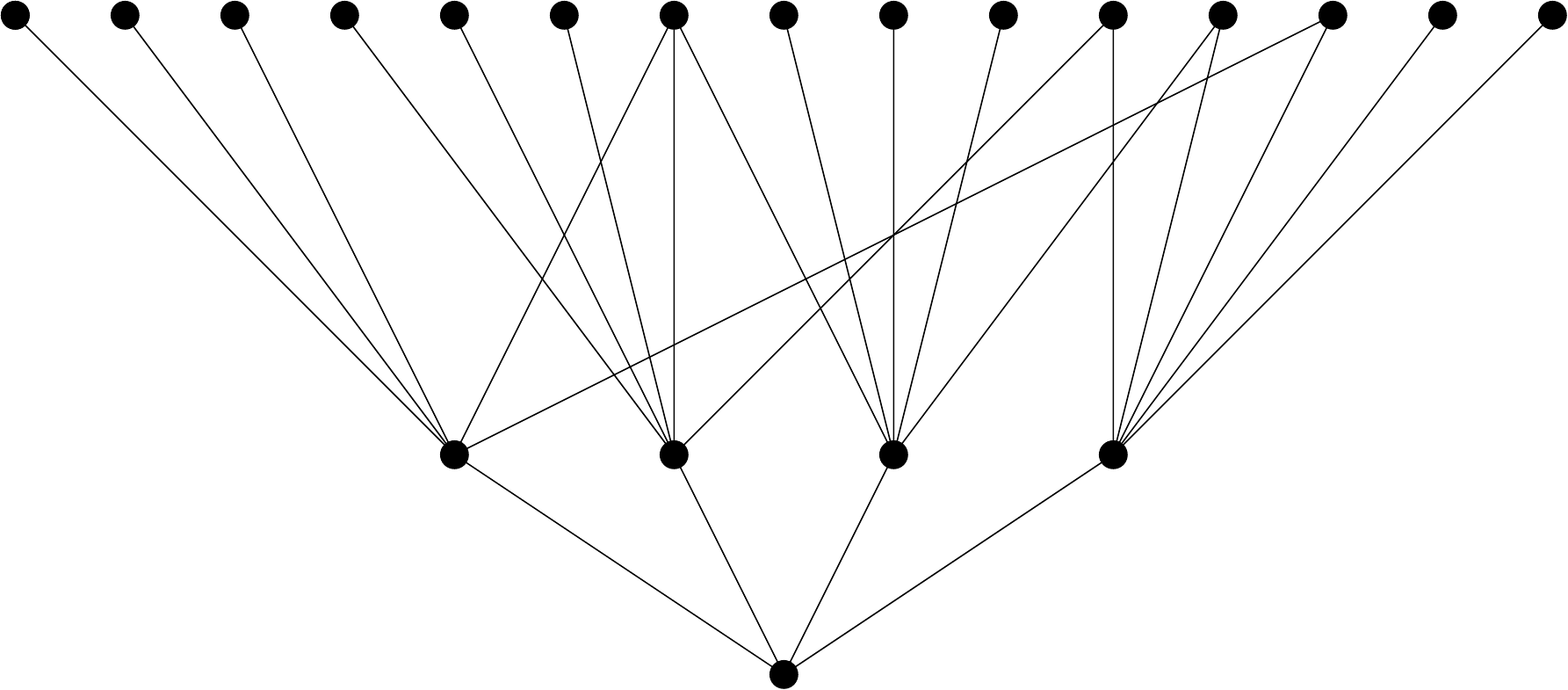} \caption{The Hasse diagram, up to rank two, of a 4-differential poset with rank function starting $1,4,15,\dots $, corresponding to the hypergraph $H $ of Example \ref{4}.}
\end{figure}

\begin{remark}\label{pp}
Since, by Proposition \ref{simp}, the  hypergraphs $H$, associated with the rank-selected subposets $P_{[1,2]}$ of our $r$-differential posets, are precisely those (with all hyperedges of positive dimension) for which every 2-subset of the vertex set is contained in exactly one hyperedge, the posets $P_{[1,2]}$ can also be given, as a special case, a  natural interpretation in terms of objects of particular interest in design theory and finite geometries. 

Recall that a \emph{Steiner system} $S(l,m,r)$ is an $r$-element set $V$, together with a collection of $m$-subsets of $V$ (called \emph{blocks}) having the property that every $l$-subset of $V$ is contained in exactly one block. Perhaps the best known family of examples  is that of \emph{finite projective planes}, which are the Steiner systems  $S(2,q+1,q^2+q+1)$, where $q\geq 2$ is called the \emph{order} of the plane. A major open problem in  geometry is a classification of all finite projective planes, or even of the possible values that $q$ can assume. (Conjecturally, $q$ may only be the power of a prime.) We refer to, e.g., \cite{De} for an introduction to this area.

It  follows from Proposition \ref{simp} that classifying the rank-selected subposets $P_{[1,2]}$ of  $r$-differential posets will imply a characterization of all Steiner systems $S(2,m,r)$ --- which correspond to the very special case where the hypergraphs associated with the $P_{[1,2]}$ are uniform --- and in particular, it will imply a characterization of all finite projective planes.
\end{remark}

\begin{corollary}\label{6}
If and only if $r\geq 6$, there exist $r$-differential posets, say $P$ and $Q$, having the same rank function, but such that all corresponding rank-selected subposets $P_{[a,a+1]}$ and $Q_{[a,a+1]}$ are nonisomorphic, for any $a>0$.
\end{corollary}

\begin{proof} Let us consider two posets $P_{[1,2]}$ and $Q_{[1,2]}$ corresponding, by Proposition \ref{simp}, to nonisomorphic hypergraphs over the same vertex set, whose sum of the dimensions of the hyperedges is the same (i.e., they give the same value of $p_2$). As we have seen in Example \ref{4}, such hypergraphs exist when $r=6$, and for no value of $r\le 5$. For any $r>6$, they can be constructed inductively as follows. Suppose that $H$ and $H'$ are admissible nonisomorphic  hypergraphs over $\{1, 2, \dots , r-1\}$. Then it is easy to see that we  obtain two admissible nonisomorphic hypergraphs over $\{1, 2, \dots , r\}$ simply by  adding, to both $H$ and $H'$, the 1-dimensional hyperedges $\{1,r\}, \{2,r\}, \dots ,\{r-1,r\}$. 

If we now construct two $r$-differential posets, $P$ and $Q$, by infinitely iterating Wagner's construction starting with the nonisomorphic posets $P_{[1,2]}$ and $Q_{[1,2]}$, it is clear that no $P_{[a,a+1]}$ is isomorphic to $Q_{[a,a+1]}$, for any $a>0$, while of course the two rank functions coincide, as desired.
\end{proof}

\begin{remark}\label{ppp}
The proof of Corollary \ref{6} implies, in light of Remark \ref{pp}, the existence of very interesting families of nonisomorphic differential posets with the same rank functions, namely those  obtained by infinitely iterating Wagner's construction starting with nonisomorphic projective planes of the same order. For instance, the smallest order $q$ allowing the existence of  nonisomorphic projective planes is  $9$, where the four possible distinct classes of isomorphism  were already known to Veblen \cite{VW}. This gives rise to four 91-differential posets having the same rank function which are nonisomorphic in any corresponding rank-selected subposets.
\end{remark}

\begin{remark}\label{alpha} 
\begin{enumerate}
\item[(i)] Given an $r$-differential poset $P$, one can associate inductively, in an entirely similar fashion, a hypergraph $H_n$, having all hyperedges of positive dimension, to any rank-selected subposet $P_{[n,n+1]}$ of $P$. The elements of $P_n$ make the vertex set $V_n$ of $H_n$, a 2-subset of $V_n$ is contained in a hyperedge of $H_n$ if and only if the corresponding elements of $P_n$ cover a common element of $P_{n-1}$, and thanks to Lemma \ref{cover},  no 2-subset of $V_n$ belongs to more than one hyperedge of $H_n$. 

Obviously however, handling the hypergraphs $H_n$ for higher values of $n$ is generally even more complicated than it is for $n=1$, in that the structure of $P_{[n,n+1]}$ also heavily depends on the structure of the previous rank-selected subposets. We only notice here that, if we let  $T_n$ denote the sum of the dimensions of the hyperedges of the hypergraph associated with $P_{[n,n+1]}$, from the simple equality
$$p_{n+1}=\sum_{F {\ }\text{hyperedge of} {\ }H_{n-1}}\left(|F|+r\right)+(r+1)\cdot \left(p_n-\sum_{F {\ }\text{hyperedge of}{\ } H_{n-1}}1\right)-T_n,$$
 by induction one  easily obtains 
\begin{equation}\label{tn}
T_n=r\sum_{j=0}^np_j-p_{n+1}.
\end{equation}

Notice that formula (\ref{tn}) generalizes the formula for $T_1$ of part (2) of Proposition \ref{simp}. Also, if following the notation of \cite{St}, we define $\alpha (n\rightarrow n+1)$ to be the number of  pairs $(x,y)\in P_{n+1}\times P_n$ such that $x$ covers $y$, one moment's thought gives that:
$$p_{n+1}=\alpha (n\rightarrow n+1)-T_n.$$
In particular, this latter equality, combined with formula (\ref{tn}), yields another proof of the following useful fact (which  will  be employed in a different context in the next section):
\begin{equation}\label{alpha1}
\alpha (n\rightarrow n+1)=r\sum_{j=0}^np_j.
\end{equation}

This result was first proved by the first author (it is essentially the case $k=1$ of equality (16) of \cite{St}, Theorem 3.2, after equating the coefficients of the generating functions).
\item[(ii)] By formula (\ref{tn}), given $p_1,\dots ,p_n$, minimizing the value of $p_{n+1}$ is tantamount to maximizing that of $T_n$. Notice, however, that the argument of Proposition \ref{simp}, part (5), cannot in general translate straightforwardly to the case $n>1$. Indeed, it is easy to see that, unlike for $H_1$, replacing a hyperedge  of cardinality $c$ of $H_n$ with hyperedges of lower cardinality --- an operation always consistent with axiom (i) of Definition \ref{diff} --- can in general violate axiom (ii).
\end{enumerate}
\end{remark}

The \emph{Interval Property} was first introduced by the second author \cite{Za} in the context of combinatorial commutative algebra, where he conjectured its existence for the set of Hilbert functions of level --- and, in a suitably symmetric fashion, Gorenstein --- algebras (we refer to \cite{Za} for the relevant definitions). This property says that if two  sequences, $h$ and $h'$, of a given class $S$ of integer sequences, coincide in all entries but one, say $h=(h_0,\dots ,h_{i-1},h_i,h_{i+1},\dots )$ and $h'=(h_0,\dots ,h_{i-1},h_i+\alpha ,h_{i+1},\dots )$ for some index $i$ and some positive integer $\alpha $, then the sequence $(h_0,\dots ,h_{i-1},h_i+\beta ,h_{i+1},\dots )$ is also in $S$, for each  $\beta =0,1,\dots , \alpha .$

The Interval Property, which is still wide open for level and Gorenstein algebras, has then also been  conjectured in \cite{BMMNZ}  for other important sequences arising in combinatorics or combinatorial algebra, namely for pure $O$-sequences and for the $f$-vectors of pure simplicial complexes. As for pure $O$-sequences, in \cite{BMMNZ} the property has been shown to hold for all sequences of length 4 (a result that also led to a proof of the first author's matroid $h$-vector conjecture in rank 3; see \cite{HSZ}), while most recently, it has been disproved in the four variable case by A. Constantinescu and M. Varbaro; for pure $f$-vectors, the property is still wide open, and little progress has been made to date. The Interval Property is well-known to hold, for instance, for  the Hilbert functions of graded algebras of any Krull dimension, for the $f$-vectors of arbitrary simplicial complexes, or for the $h$-vectors  of Cohen-Macaulay (or shellable) simplicial complexes, while it fails, e.g., for matroid $h$-vectors (see \cite{BMMNZ,St3} for details). In the contexts where the Interval Property has been conjectured, an explicit characterization of the desired sequences seems to be out of reach, and although it is often very challenging to prove or disprove, when confirmed it appears to be both a consequential property and one of the strongest structural results that one might hope to achieve for the set of such sequences. 

Notice that  rank functions of differential posets can  be  identified with Hilbert functions of certain graded vector spaces (see \cite{St}, Section 2), and that, as we have seen earlier in this section, an explicit characterization of the set $\mathbb X$ of all possible rank functions of differential posets seems  nearly impossible. Therefore it would be extremely interesting if the Interval Property might somehow extend to this context. Unfortunately however, the following result shows  that, at least in general,  rank functions of differential posets do not enjoy the Interval Property. (On the other hand, see Question \ref{5} below and the comment following it.)

\begin{theorem}\label{ip}
The Interval Property does not hold for the set $\mathbb X$. More precisely, the two sequences
$$\textbf{p}'{\ }: {\ }1,4,14,p_3=60,p_4=254,p_5,p_6,\dots $$
and
$$\textbf{p}''{\ }: {\ }1,4,17,p_3=60,p_4=254,p_5,p_6,\dots ,$$
where $p_i=4p_{i-1}+p_{i-2}$ for all $i\geq 5$, belong to $\mathbb X$, while $$\textbf{p}'''{\ }: {\ }1,4,16,p_3=60,p_4=254,p_5,p_6,\dots $$
does not.
\end{theorem}

\begin{proof} In order to show that \emph{$\textbf{p}'\in \mathbb X$}, it suffices to consider a differential poset whose rank function begins $1,4,14,\dots $, which exists from Proposition \ref{simp} (see also Example \ref{4}). Then, by infinitely iterating Wagner's construction starting in rank two, we clearly obtain a differential poset with rank function \emph{$\textbf{p}'$}. That \emph{$\textbf{p}'''\notin  \mathbb X$} follows from the fact that no rank function of a differential poset can begin $1,4,16,\dots $ (see again Proposition \ref{simp} and Example \ref{4}). Therefore, it remains to construct a differential poset having rank function \emph{$\textbf{p}''$}.

In order to do this, let us start by considering the  Fibonacci $4$-differential poset, $Z(4)$, up to rank three. Its rank function is $1,4,17,72$. We want to replace a suitable portion of the Hasse diagram of $Z(4)$ between ranks two and three with another diagram so that the resulting poset has the same elements in rank two, and 12 fewer elements in rank three. Notice that the new Hasse diagram thus constructed will be the Hasse diagram of some differential poset up to rank three if and only if, for any  element of rank two, the number of elements of rank three covering it is unchanged (using axiom (ii) of Definition \ref{diff}), and two elements of rank two  have some common cover  if and only if they had some common cover in $Z(4)$ (using axiom (i)). 

How to perform this operation is described in Figure 2, where the first diagram is the portion of $Z(4)$ being removed, and the second diagram is its replacement. (The convention adopted in Figures 2 and 3 is that the elements of higher rank do not cover other elements outside of those displayed in the diagrams, while the elements of lower rank may also be covered by other elements.)
\begin{figure}[h!t]
\includegraphics[scale=0.44]{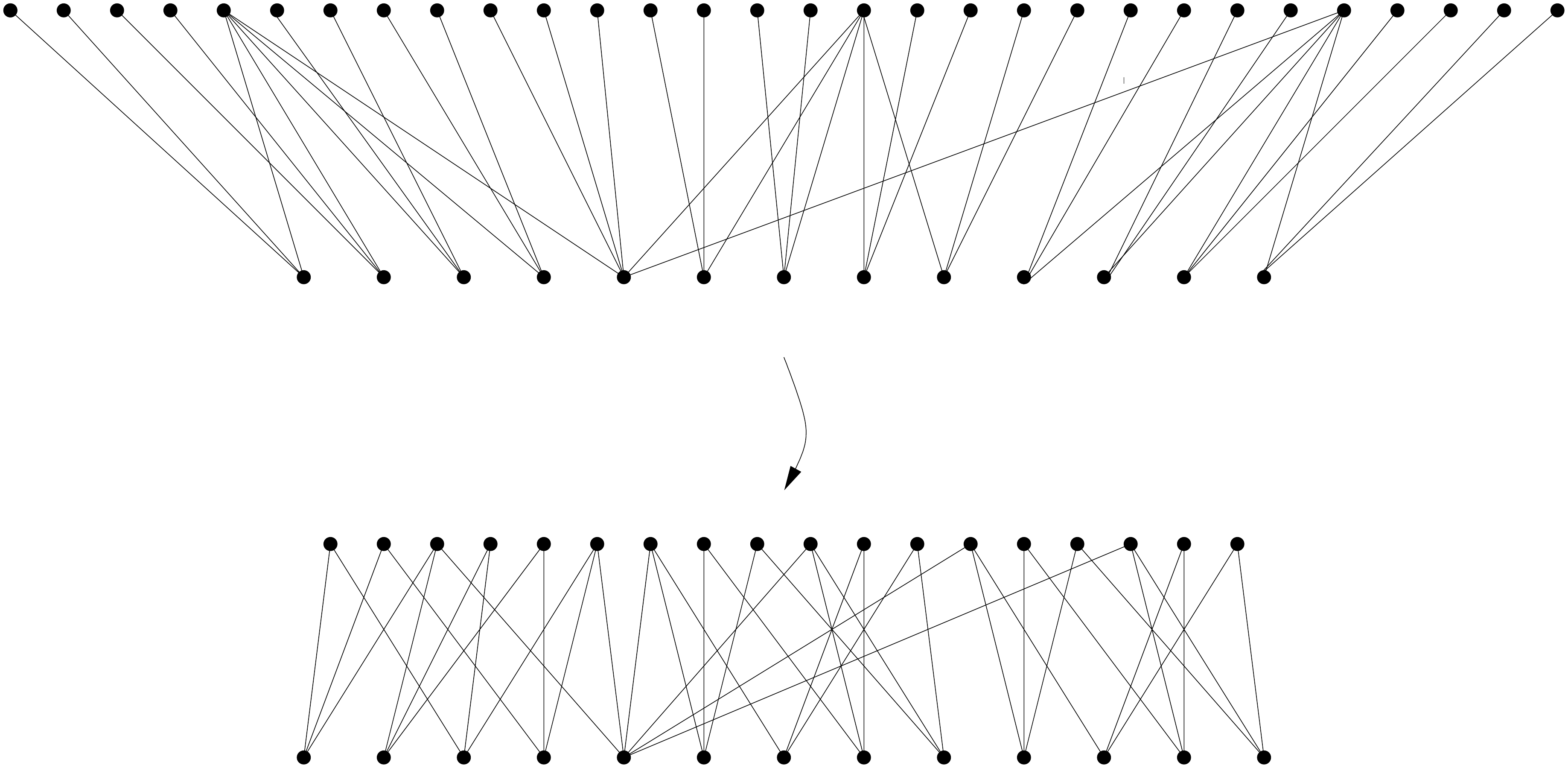} \caption{The Hasse diagram of the portion of $Z(4)$ being replaced in Theorem \ref{ip} (above), and that of its replacement (below).}
\end{figure}

Hence we have now obtained a new differential poset up to rank three, with rank function  $1,4,17,60$. By Wagner's construction, this poset can be extended by one rank to a differential poset up to rank four, say $F$, having rank function  $1,4,17,60, 257$. If, similarly, we can now replace a portion of the Hasse diagram of $F$ between ranks three and four and obtain a new differential poset up to rank four, say $G$, having 3 fewer elements of rank four than $F$, then we are done. Indeed, $G$ will have rank function  $1,4,17,60, 254$, hence  by infinitely iterating Wagner's construction we will get a differential poset with rank function \emph{$\textbf{p}''$}, as desired.

The operation transforming $F$ into $G$ in ranks three and four is described in Figure 3; notice that the first diagram does indeed represent a portion of the Hasse diagram of $F$, as it is easy to check. This completes the proof of the theorem.
\end{proof}
{\ }\\
\begin{figure}[h!t]
\includegraphics[scale=0.485]{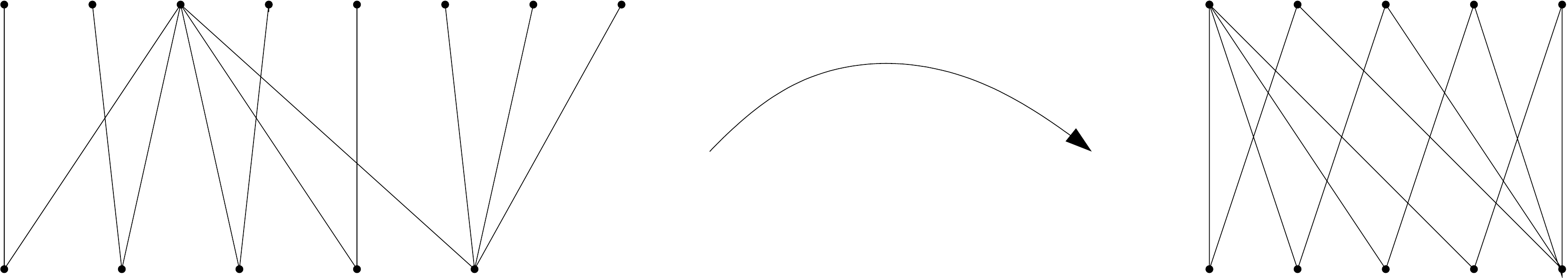} \caption{The Hasse diagram of the portion of the differential poset $F$ being replaced in Theorem \ref{ip} (left), and that of its replacement (right), giving rise to $G$.}
\end{figure}

\section{The nonpolynomial growth of $p_n$}

The object of this section is to prove that the rank function  of any differential poset has nonpolynomial growth. Recall the well-known Hardy-Ramanujan asymptotic formula for the partition function \cite{HR},
$$p(n)\sim \frac{1}{4n\sqrt{3}}e^{\pi\sqrt{2n/3}}.$$
 
More generally, for any given $r\geq 1$, the rank function $p_r(n)$ of $Y^r$ is clearly determined by the functional equation  $\sum_{n\geq 0} p_r(n)q^n=(\sum_{n\geq 0} p(n)q^n)^r$.  Using a classical theorem of asymptotic combinatorics due to Meinardus \cite{Me} and some basic properties of the Riemann zeta function, it is then a fairly standard task, of which we omit the details, to prove that
$$p_r(n)\sim Cn^{\alpha }e^{\pi\sqrt{2rn/3}},$$
for some suitable constants $C$ and $\alpha$.

The main goal of this section is to show that, for any arbitrary $r$-differential poset, its rank function  satisfies the asymptotic lower bound 
$$p_n\gg n^ae^{2\sqrt{rn}},$$
for some constant $a$. Notice that $\pi\sqrt{2/3}=2.56...$, which implies that our result is very close to being optimal, and in the much greater level of generality of differential posets. Notice also that the above asymptotic formula for $p(n)$ was first proved by Hardy and Ramanujan using modular forms and the circle method of additive number theory, while in a famous  Annals of Math. paper \cite{Er}, Erd\H{o}s was able to reprove the Hardy-Ramanujan formula, up to a multiplicative constant,  without making use of complex analytic techniques.

Let $P=\bigcup_{n\geq 0}P_n$ be any  $r$-differential poset with rank function \emph{$\textbf{p}{\ }: {\ }p_0=1,p_1=r,p_2,\dots $}. Following  the notation of \cite{St}, define $\kappa (n \rightarrow n+1 \rightarrow n \rightarrow n+1 \rightarrow n)$ to be the number of closed Hasse walks $x_1<x_2>x_3<x_4>x_5=x_1$, where $x_1,x_3\in P_n$ and $x_2,x_4\in P_{n+1}$. Also, denote by $\alpha (n \rightarrow n+1 \rightarrow n)$  the number of Hasse walks $x_1<x_2>x_3$, where $x_1,x_3\in P_n$ and $x_2\in P_{n+1}$, and let $\alpha (n \rightarrow n+1)$, as in the previous section, be the number of Hasse walks $x_1<x_2$, where $x_1\in P_n$ and $x_2\in P_{n+1}$. Finally, for any $x\in P_n$, denote by $c(x)$ the number of covers of $x$. 

Our first key lemma is:

\begin{lemma}\label{kappa}
For any $r$-differential poset $P$ and any nonnegative integer $n$,
$$\sum_{x\in P_n}c(x)^2=\kappa (n \rightarrow n+1 \rightarrow n \rightarrow n+1 \rightarrow n)-\alpha (n \rightarrow n+1 \rightarrow n)+\alpha (n \rightarrow n+1).$$
\end{lemma}

\begin{proof}
Among the closed Hasse walks in $P$ enumerated by $\kappa (n \rightarrow n+1 \rightarrow n \rightarrow n+1 \rightarrow n)$, it is easy to see that $c(x)^2$ is the number of those walks starting at $x=x_1\in P_n$ such that $x_3=x_1$. Otherwise, if $x_3\neq x_1$,  clearly $x_2$ is a common cover of $x_1$ and $x_3$. Similarly, from the uniqueness of the common cover, it follows that $x_4= x_2$. 

Therefore one moment's thought now gives that, in order to obtain $\kappa (n \rightarrow n+1 \rightarrow n \rightarrow n+1 \rightarrow n)$, we need to add to $\sum_{x\in P_n}c(x)^2$ a contribution of exactly
$$\alpha (n \rightarrow n+1 \rightarrow n)-\sum_{x\in P_n}c(x)=\alpha (n \rightarrow n+1 \rightarrow n)-\alpha (n \rightarrow n+1),$$
which concludes the proof of the lemma.
\end{proof}

The following lemma is also of independent interest, and appears to be new even for the Young lattice (so when $r=1$ and $p_n=p(n)$, the partition function).

\begin{lemma}\label{c2}
For any $r$-differential poset $P$ with with rank function $\textbf{p}{\ }: {\ }p_0,p_1,p_2,\dots $ and  any nonnegative integer $n$, we have:
$$\sum_{x\in P_n}c(x)^2=\sum_{j=0}^n (r^2(n-j+1)+\epsilon r)p_j,$$
where $\epsilon $ is the remainder of $n-j$ modulo 2.
\end{lemma}

\begin{proof}
We will make extensive use, in this proof, of the notation of \cite{St}. We begin by evaluating $\kappa (n \rightarrow n+1 \rightarrow n \rightarrow n+1 \rightarrow n)$. By \cite{St}, Theorem 3.12, we easily have that, if we set $F(q)=\sum_{n\geq 0}p_nq^n$, then $\kappa (n \rightarrow n+1 \rightarrow n \rightarrow n+1 \rightarrow n)$ is the degree $n$ coefficient of $F(q)B_f(q)$, where
$$f=f(U,D)=(DU)^2=U^2D^2+3rUD+r^2,$$
and
\begin{equation}\label{bf}
B_f(q)=\frac{2r^2q^2}{(1-q)^2}+\frac{3r^2q}{1-q}+r^2.
\end{equation}

An equally standard computation gives that $\alpha (n \rightarrow n+1 \rightarrow n)$ is the  degree $n$ coefficient of $F(q)A_w(q)$ in Corollary 3.4 of \cite{St}. In the notation of that corollary, the word $w$ is $w=DU$, $g_U(z)=z$, and $$g_w(z)=(r+z)g_U(z)+r\frac{dg_U(z)}{dz}=(r+z)z+r\cdot 1=z^2+rz+r.$$

Therefore, using the rational functions $A_k(q)$ of \cite{St}, Theorem 3.2 (see page 930 for some explicit computations), $A_w(q)$ can be written as:
\begin{equation}\label{aw}
A_w(q)= \frac{r(r+1)q^2+r(r-1)q^3}{(1-q)(1-q^2)}+\frac{r^2q}{1-q}+r.
\end{equation}

Finally, by \cite{St}, Theorem 3.2, or by our equality (\ref{alpha1}) above, we have that 
$\alpha (n \rightarrow n+1 )$ is the degree $n$ coefficient of $F(q)A_1(q)$, where
\begin{equation}\label{a1}
A_1(q)=\frac{r}{1-q}.
\end{equation}

Hence, employing Lemma \ref{kappa} and equations (\ref{bf}), (\ref{aw}) and (\ref{a1}), we easily obtain that $\sum_{x\in P_n}c(x)^2$ is  the  degree $n$ coefficient of:

$$F(q)\left(\frac{2r^2q^2}{(1-q)^2}+\frac{3r^2q}{1-q}+r^2-\frac{r(r+1)q^2+r(r-1)q^3}{(1-q)(1-q^2)}-\frac{r^2q}{1-q}-r+\frac{r}{1-q}\right)$$
\begin{equation}\label{fq}
=F(q)\frac{-rq^2+(r^2+r)q+r^2}{(1-q)(1-q^2)}.
\end{equation}

Notice that if $\lceil a \rceil $ as usual denotes the smallest integer $\geq a$, then we have:
$$\frac{F(q)}{(1-q)(1-q^2)}=F(q)\left(\sum_{i\geq 0}q^i\right)\left(\sum_{i\geq 0}q^{2i}\right)$$$$=\sum_{n\geq 0}p_nq^n\cdot \sum_{i\geq 0}\left\lceil \frac{i+1}{2}\right\rceil q^i=\sum_{n\geq 0}\left(\sum_{j=0}^n\left\lceil \frac{n-j+1}{2}\right\rceil p_j\right)q^n.$$

Thus, from (\ref{fq}), it follows that $\sum_{x\in P_n}c(x)^2$ is  $r^2$ times the  degree $n$ coefficient of the right hand side of the last equation, plus $r^2+r$ times its  degree $n-1$ coefficient, minus $r$ times its  degree $n-2$ coefficient. 

Therefore,
$$\sum_{x\in P_n}c(x)^2=r^2 \sum_{j=0}^n\left\lceil \frac{n-j+1}{2}\right\rceil p_j +(r^2+r)\sum_{j=0}^{n-1}\left\lceil \frac{n-j}{2}\right\rceil  p_j-r\sum_{j=0}^{n-2}\left\lceil \frac{n-j-1}{2}\right\rceil p_j,$$
which is a simple exercise to check that it coincides with the formula in the statement. This completes the proof of the lemma.
\end{proof}

Finally, we need an asymptotic estimate for $\alpha(0\rightarrow n)$, the number of chains in $P$ from rank 0 to rank $n$.

\begin{lemma}\label{0n}
For any fixed positive integer $r$ and any $r$-differential poset $P$, we have:
$$\alpha(0\rightarrow n)\sim \frac{\sqrt{n!}\cdot r^{n/2}e^{\sqrt{rn}}}{(8\pi e^{3r-2}n)^{1/4}}.$$
\end{lemma}

\begin{proof}
From the exponential generating function of \cite{St}, Proposition 3.1 (formula (12)), namely
$$\sum_{n\geq 0}\alpha(0\rightarrow n)\frac{q^n}{n!}=e^{rq+rq^2/2},$$
it is a standard task to determine  estimates for $\frac{\alpha(0\rightarrow n)}{n!}$, and in particular the desired asymptotic value for $\alpha(0\rightarrow n)$. (See e.g. \cite{Pe}, Example 3.2, where the case $r=1$ is worked out in detail, using the saddle point method.)
\end{proof}

Given two functions $f,g: \mathbb N \longrightarrow \mathbb R^{+}$, we say that $f(n)\gg g(n)$ if $f(n)\ge Cg(n)$ for $n$ large, for some positive constant $C$ (i.e., if $g(n)=O(f(n))$).

\begin{theorem}\label{main}
Let $P$ be any arbitrary $r$-differential poset with rank function $\textbf{p}{\ }: {\ }p_0,p_1,p_2,\dots $. Then there exists a constant $a$ such that
$$p_n\gg n^ae^{2\sqrt{rn}}.$$
\end{theorem}

\begin{proof}
For any $x\in P_n$, let $e(x)$ be the number of maximal chains in $P$ ending at $x$. Hence, clearly, $\alpha(0\rightarrow n)=\sum_{x\in P_n}e(x)$, and \begin{equation}\label{ddd}
\alpha(0\rightarrow n+1)=\sum_{x\in P_n}c(x)e(x).
\end{equation}

Thus, from the Cauchy-Schwarz inequality, we obtain
\begin{equation}\label{cs}
\sum_{x\in P_n}c(x)^2\geq \frac{\left(\sum_{x\in P_n}c(x)e(x)\right)^2}{ \sum_{x\in P_n}e(x)^2}.
\end{equation}

By \cite{St}, Corollary 3.9, we have that $\sum_{x\in P_n}e(x)^2=r^nn!$. Also, since by  \cite{St}, Corollary 4.3, the $p_j$ are nondecreasing, Lemma \ref{c2}   yields the trivial estimate
$$\sum_{x\in P_n}c(x)^2\ll \sum_{j=0}^nnp_n\ll n^2p_n.$$

Therefore, from (\ref{ddd}), (\ref{cs}) and Lemma \ref{0n}, we have
$$n^2p_n\gg \frac{\alpha(0\rightarrow n+1)^2}{r^nn!}\gg \frac{n^{a'}(n+1)!\cdot r^{(n+1)}e^{2\sqrt{r(n+1)}}}{r^nn!},$$
and the theorem  follows.
\end{proof}

\section{Open problems}

We wrap up this note by gathering together a few open questions specifically on the rank function of a differential poset, some of which were previously posed in \cite{St}.

As we saw in the last section, it is well known from \cite{St}, Corollary 4.3, that the rank function $p_n$ of any differential poset is \emph{weakly} increasing. However, it is reasonable to conjecture that, in fact,  strict inequality always holds (see also \cite{MR,St}).

\begin{question}\label{1}
Is it true that, for any differential poset and any positive integer $n$, $p_n< p_{n+1}$?
\end{question} 

Note that the inequality $p_n\leq p_{n+1}$ was proved in \cite{St} using a linear algebra argument. Determining  a combinatorial proof might very well help answer positively Question \ref{1}. 

Define the \emph{first difference} \emph{$\Delta \textbf{p}$} of a sequence \emph{$\textbf{p}{\ }: {\ } 1,p_1,p_2,\dots $} as the sequence \emph{$\Delta \textbf{p}{\ }: {\ }1,(\Delta \textbf{p})_1=p_1-1,(\Delta \textbf{p})_2=p_2-p_1,\dots $}. Recursively, define the \emph{$t$-th difference} $\Delta^t$ as the first difference of the $(t-1)$-st difference. Motivated by all examples we are aware of (and the name itself of a differential poset!), more boldly we ask:

\begin{question}\label{2} (i) Is it true that, for any $r$-differential poset with rank function \emph{$\textbf{p}$}, \emph{$\Delta^r \textbf{p}>0$} in any ``rank''  $n\geq 2$?\\
(ii) Is it true that, for any differential poset with rank function \emph{$\textbf{p}$} and any integer $t$, \emph{$\Delta^t \textbf{p}>0$} in any ``rank''  $n$ large enough (depending on $t$)?
\end{question}

Notice that Question \ref{2} is known to have an affirmative answer for both the Young lattice and the Fibonacci $r$-differential posets. As for $Z(r)$, both parts (i) and (ii) are easy to check, and so is part (i) for $Y$. Part (ii) for $Y$ was proved by H. Gupta \cite{Gu}, in response to a problem raised by G. Andrews (see also A.M. Odlyzko \cite{Od}, for a  sharper estimate using the Rademacher convergent series expansion for $p(n)$). Of course, for 1-differential posets, Question \ref{1} is equivalent to  Question \ref{2}, part (i). See also Miller-Reiner's \cite{MR}, Conjecture 2.3, for a related conjecture.

Conversely, we  ask (see also \cite{St}):

\begin{question}\label{22}
Is it true that, for any $r$-differential poset and any integer $n\geq 2$,
$$p_n\leq rp_{n-1}+p_{n-2}?$$
\end{question}

Answering a question posed in \cite{St}, P. Byrnes \cite{B} recently proved that the rank function of any $r$-differential poset is bounded from above, rankwise, by that of  $Z(r)$, which satisfies the equality in Question \ref{22} for all $n$. 

As for a lower bound, in \cite{St} the first author  asked the following:

\begin{question}\label{3}
Is the rank function of any $r$-differential poset bounded from below, rankwise, by that of $Y^r$?
\end{question}

For both Questions \ref{22} and \ref{3}, see part (ii) of Remark \ref{alpha}. Note that, even though the asymptotic lower bound we have proved for $p_n$ in the previous section is  close to being optimal, our approach does not appear to be of help in settling Question \ref{3}.

As we have seen,  a characterization of all rank functions of differential posets seems entirely out of reach. Thus, even  though the Interval Property has been shown to fail in general in Theorem \ref{ip}, it would be  helpful to recover it for some important families of  differential posets, for it would at least imply a natural and very strong  structural result on the set of their rank functions. For instance, we ask:

\begin{question}\label{5}
Do 1-differential posets enjoy the Interval Property?
\end{question}

P. Byrnes \cite{B} has computed the sequences that may occur as rank functions of 1-differential posets up to rank ten (there are 44,606 nonisomorphic 1-differential posets up to rank nine, and 29,199,636 up to rank ten). Such sequences suggest that, at least for the initial ranks, Question \ref{5} could very well have a positive answer.\\
\\
\textbf{Note added after acceptance.} By means of a nice algebraic argument, A. Miller \cite{Mil} recently proved that Question 4.1 has a positive answer; that is, the rank function of any $r$-differential poset is \emph{strictly} increasing, with the only exception $p_0=p_1$ when $r=1$. It remains a very interesting (and possibly consequential) open problem to determine a combinatorial proof of this result.

\subsection*{Acknowledgements} The second author warmly thanks the first author for his terrific hospitality during calendar year 2011, the MIT Mathematics Department for partial financial support, and Dr. Mark Gockenbach and the Michigan Tech Mathematics Department, from which the second author was on partial leave, for extra summer support. We also thank David Cook II and Joel Lewis for comments, and Pat Byrnes for kindly sharing an early draft of the results on differential posets that will be part of his Thesis \cite{B}.\\


\end{document}